\documentclass[11pt,reqno]{amsart}

\usepackage{graphicx}
\usepackage{epstopdf}
\usepackage{amssymb,amsmath}
\title[Variational methods for the solution of FDSL/FCSL problems]{Variational methods for the solution of fractional discrete/continuous Sturm--Liouville problems}

\author{Ricardo Almeida}
\address{Center for Research and Development in Mathematics and Applications (CIDMA)\\
Department of Mathematics, University of Aveiro\\
Aveiro, 3810-193\\
Portugal}
\email{ricardo.almeida@ua.pt}

\author{Agnieszka B. Malinowska}
\address{Faculty of Computer Science, Bialystok University of Technology\\
Wiejska 45A\\
Bia\l ystok, 15-351\\
Poland}
\email{a.malinowska@pb.edu.pl}

\author{M. Lu\'{i}sa Morgado}
\address{Department of Mathematics, University of Tr\'{a}s-Os-Montes e Alto Douro,UTAD\\
 5000-801, Vila Real\\
  Portugal}
\email{luisam@utad.pt}

\author{Tatiana Odzijewicz}
\address{Department of Mathematics and Mathematical Economics\\
Warsaw School of Economics\\
Al. Niepodleg\l o\'sci 162\\
Warsaw, 02-554\\
Poland}
\email{tatiana.odzijewicz@sgh.waw.pl}


\numberwithin{equation}{section}

\newtheorem{example}{Example}[section]  
\newtheorem{theorem}[example]{Theorem}  
\newtheorem{definition}[example]{Definition}
\newtheorem{lemma}[example]{Lemma}
\newtheorem{proposition}[example]{Proposition}

\newcommand{\R}{\mathbb{R}}

\newcommand{\fonction}[5]{\begin{array}[t]{lrcl}#1 :&#2 &\longrightarrow &#3\\&#4& \longmapsto &#5 \end{array}}

\begin{document} 

\begin{abstract}
The fractional Sturm--Liouville eigenvalue problem appears in many situations, e.g., while solving
anomalous diffusion equations coming from physical and engineering applications. Therefore to obtain solutions or
approximation of solutions to this problem is of great importance. Here, we describe how
the fractional Sturm--Liouville eigenvalue problem can be formulated as a constrained fractional variational principle and show how such formulation can be used in order to approximate the solutions. Numerical examples are given, to illustrate the method.
\end{abstract}

\maketitle 

\section{Introduction}\label{intro}

Fractional calculus is a mathematical approach dealing with integral and differential terms of non-integer order. The concept of fractional calculus appeared shortly after calculus itself, but the development of practical applications proceeded very slowly. Only during the last decades, fractional problems have increasingly attracted the attention of many researchers. Applications of fractional operators include chaotic dynamics \cite{Zaslavsky},
material sciences \cite{book:Mainardi}, mechanics of fractal and complex media \cite{Carpinteri,Li}, quantum mechanics \cite{book:Hilfer}, physical kinetics \cite{Edelman} and many others (see e.g.,\cite{Domek,Tarasov}). Fractional derivatives are nonlocal operators and therefore successfully applied in the study of nonlocal or time-dependent processes \cite{book:Podlubny}. The well-established application of fractional calculus in physics is in the framework of anomalous diffusion behavior \cite{Bl,Chen,ovido,Leonenko,MMM,Metzler}: large jumps in space are modeled
by space-fractional derivatives of order between 1 and 2, while long waiting times are modeled by the time derivatives of order between 0 and 1. These partial fractional differential equations can be solved by the method of separating variables, which leads to the Sturm--Liouville and the Cauchy equations. It means that, if we are able to solve the fractional Sturm--Liouville problem and the Cauchy problem, then
we can find a solution to the fractional diffusion equation. In this paper, we consider two basic approaches to the fractional Sturm--Liouville problem: discrete and continuous. In both cases, we note that the problem can be formulated as a constrained fractional variational principle. A fractional variational problem consists in finding the extremizer of a functional that depends on fractional derivatives (differences) subject to boundary conditions and possibly some extra constraints. It is worthy to point out that the fractional calculus of variations has itself remarkable applications
in classical mechanics. Riewe \cite{CD:Riewe:1996,CD:Riewe:1997} showed that a Lagrangian involving fractional time derivatives leads to an equation of motion with non-conservative forces such as friction. For more about the fractional calculus of variations we refer the reader to \cite{cap1:book:frac:ICP2,cap1:book:Klimek,cap1:book:AD,cap1:book:ADT} while for various approaches to fractional Sturm--Liouville problems we refer to \cite{AlM09,QMQ,Tatiana1,K15,K16,zay}.

The paper is divided into two main parts dedicated, respectively, to discrete (Section~\ref{sec:DFC}) and continuous (Section~\ref{sec:CFC}) fractional problems.  In the first part we give a constructive proof of the existence of orthogonal solutions to the discrete fractional Sturm--Liouville eigenvalue problem (Theorem~\ref{existence_SL}), and show that the smallest and largest eigenvalues can be characterized as the optimal values of certain functionals (Theorem~\ref{th:VP} and Theorem~\ref{thm:RQ}). Our results are illustrated by an example. In the second part we recall the fractional variational principle and the spectral theorem for the continuous fractional Sturm--Liouville problem. Since for most problems involving fractional derivatives (equations or variational problems) one cannot provide methods to compute the exact solutions analytically, numerical methods should be used for solving such problems. Discretizing both the fractional Sturm--Liouville equation and related with it isoperimetric variational problem we show, by an example, how the variational method can be used for solving the fractional Sturm--Liouville problem.

\section{Discrete Fractional Calculus}\label{sec:DFC}

In this section we explain a relationship between the fractional Sturm--Liouville difference problem and a constrained discrete fractional variational principle. Namely, it is possible to look for solutions of Sturm--Liouville fractional difference equations by solving finite dimensional constrained optimization problems.
We shall start with necessary preliminaries. There are various versions of the fractional differences, we can mention here those introduced by Diaz and Osler \cite{diaz}, Miller and Ross \cite{miller}, Atici and Eloe \cite{Atici1,Atici2} or the Caputo difference \cite{abd}. In this paper, we use the notion of  Gr\"{u}nwald--Letnikov \cite{kaczorek,book:Podlubny}.

Let us define the mesh points
$x_j = a + jh,\quad j = 0, 1,\ldots, N,$
where $h$ denotes the uniform space step and set $D=\{x_0,\ldots,x_N\}$. In what follows $\alpha\in \R$ and $0<\alpha\leq 1$. Moreover, we set
\begin{equation}\label{coef}
a_i^{(\alpha)}:=
\begin{cases}
1, & \text{ if } i=0\\
(-1)^i\frac{\alpha (\alpha-1)\cdots (\alpha-i+1)}{i!}, & \text{ if } i=1,2,\ldots.
\end{cases}
\end{equation}

\begin{definition}
The backward fractional difference of order $\alpha$, where $0<\alpha\leq1$, of function $f:D\rightarrow \R$ is defined by
\begin{equation}\label{fb}
{_{0}}\Delta^{\alpha}_{k}f(x_k)
:=\frac{1}{h^{\alpha}}\sum_{i=0}^{k}(-1)^i\frac{\alpha (\alpha-1)\cdots (\alpha-i+1)}{i!}f(x_{k-i}).
\end{equation}
while
\begin{equation}\label{ff}
{_{k}}\Delta^{\alpha}_{N}f(x_k)
:=\frac{1}{h^{\alpha}}\sum_{i=0}^{N-k}(-1)^i\frac{\alpha (\alpha-1)\cdots (\alpha-i+1)}{i!}f(x_{k+i})
\end{equation}
is the forward fractional difference of function $f$.
\end{definition}

Fractional backward and forward differences are linear operators.

\begin{theorem}(cf. \cite{ostalczyk})
Let $f,g$ be two real functions defined on $D$ and $\beta,\gamma\in \R$. Then
$${_{0}}\Delta^{\alpha}_{k}[\gamma f(x_k)+\beta g(x_k)]
=\gamma {_{0}}\Delta^{\alpha}_{k}f(x_k)+\beta {_{0}}\Delta^{\alpha}_{k}g(x_k),$$
$${_{k}}\Delta^{\alpha}_{N}[\gamma f(x_k)+\beta g(x_k)]
=\gamma {_{k}}\Delta^{\alpha}_{N}f(x_k)+\beta {_{k}}\Delta^{\alpha}_{N}g(x_k),$$
for all $k$.
\end{theorem}

The following formula of the summation by parts for fractional operators will be essential for proving results concerning variational problems.

\begin{lemma}(cf. \cite{loic1})
 Let $f$, $g$ be two real functions defined on $D$. Then
 \begin{equation*}
\sum_{k=0}^{N}g(x_k){_{0}}\Delta^{\alpha}_{k}f(x_k)=\sum_{k=0}^{N}f(x_k){_{k}}\Delta^{\alpha}_{N}g(x_k).
\end{equation*}
If $f(x_0)=f(x_N)=0$ or $g(x_0)=g(x_N)=0$, then
\begin{equation}\label{int_by_parts}
\sum_{k=1}^{N}g(x_k){_{0}}\Delta^{\alpha}_{k}f(x_k)=\sum_{k=0}^{N-1}f(x_k){_{k}}\Delta^{\alpha}_{N}g(x_k).
\end{equation}
\end{lemma}

\subsection{The Sturm--Liouville Problem}
\label{SLProblem}
 In this subsection our topic is the Sturm--Liouville fractional difference equation:
 \begin{equation}\label{SL_eq}
{_{k}}\Delta^{\alpha}_{N}\left(p(x_k){_{0}}\Delta^{\alpha}_{k}y(x_k)\right)+q(x_k)y(x_k)=\lambda r(x_k)y(x_k), \quad k=1,\ldots,N-1,
\end{equation}
with boundary conditions:
\begin{equation}\label{SL_eq_bc}
y(x_0)=0,\quad y(x_N)=0.
\end{equation}
We assume that $p(x_i)>0$, $r(x_i)>0$, $q(x_i)$ is defined and real valued for all $x_i$, $i=0,\ldots,N$, and $\lambda$ is a parameter. It is required to find the eigenfunctions and the eigenvalues of the given boundary value problem, i.e., the nontrivial solutions of \eqref{SL_eq}--\eqref{SL_eq_bc} and the corresponding values of the parameter $\lambda$. Theorem below gives an answer to this question.

\begin{theorem}\label{existence_SL}
The Sturm--Liouville problem \eqref{SL_eq}--\eqref{SL_eq_bc} has $N-1$ real eigenvalues, which we denote by
$$\lambda_1\leq \lambda_2\leq\cdots \leq \lambda_{N-1}.$$
The corresponding eigenfunctions, $$y^1,y^2,\ldots,y^{N-1}:\{x_1,\ldots,x_{N-1}\}\rightarrow \R,$$ are mutually orthogonal: if $i\neq j$, then
$$\langle y^i,y^j\rangle_{r}:=\sum_{k=1}^{N-1}r(x_k)y^i(x_{k})y^j(x_{k})=0,$$
and they span $\R^{N-1}$: any vector $\varphi=\left(\varphi(x_{k})\right)_{k=1}^{N-1}\in \R^{N-1}$ has a unique expansion
$$\varphi(x_{k})=\sum_{i=1}^{N-1}c_iy^i(x_{k}),\quad 1\leq k\leq N-1.$$
The coefficients $c_i$ are given by
$$c_i=\frac{\langle \varphi,y^i\rangle_{r}}{\langle y^i,y^i\rangle_{r}}.$$
\end{theorem}

\begin{proof}
Observe that equations \eqref{SL_eq}--\eqref{SL_eq_bc} can be considered as a system of $N-1$ linear equations with $N-1$ real unknowns $y(x_1),\ldots,y(x_{N-1})$. The corresponding matrix form is as follows:
\begin{equation}\label{M_sl}
Ay^T=\lambda Ry^T,
\end{equation}
where the entries $A_{ij}$ of $A$ are
\begin{equation*}
A_{ij}^{(\alpha)}=
\begin{cases}
\frac{1}{h^{2\alpha}}\left[q(x_i)+\sum_{k=0}^{N-i}(a_k^{(\alpha)})^2p(x_{i+k})\right], \quad i=j\\
\frac{1}{h^{2\alpha}}\left[\sum_{k=0}^{N-i}a_k^{(\alpha)}p(x_{i+k})\sum_{m=0}^{k+i}a_m^{(\alpha)}\right] \text{ and } k-m+i=j,\quad i\neq j.
\end{cases}
\end{equation*}
and $R=diag\{r(x_1),\ldots,r(x_{N-1})\}$.
Writing \eqref{M_sl} as
\begin{equation}\label{M2_sl}
R^{-1}Ay^T=\lambda y^T
\end{equation}
we get an eigenvalue problem with the symmetric matrix $R^{-1}A$. Because of the equivalence of problem \eqref{SL_eq}--\eqref{SL_eq_bc} with problem \eqref{M2_sl} it follows from matrix theory that the Sturm--Liouville problem \eqref{SL_eq}--\eqref{SL_eq_bc} has $N-1$ linearly pairwise orthogonal real independent eigenfunctions with all eigenvalues real. Now we would like to find constants $c_1,\ldots,c_{N-1}$ such that $\varphi(x_{k})=\sum_{i=1}^{N-1}c_iy^i(x_{k})$, $1\leq k\leq N-1$. Note that
\begin{equation*}
\langle \varphi,y^j\rangle_{r}=\langle \sum_{i=1}^{N-1}c_iy^i,y^j\rangle_{r}
= \sum_{i=1}^{N-1}c_i\langle y^i,y^j\rangle_{r}=c_j\langle y^j,y^j\rangle_{r}
\end{equation*}
because of orthogonality. Therefore $c_i=\frac{\langle \varphi,y^i\rangle_{r}}{\langle y^i,y^i\rangle_{r}}$, $1\leq i\leq N-1$.
\end{proof}

\subsection{Isoperimetric Variational Problems}
\label{IVProblem}

In this section we prove two theorems connecting the Sturm--Liouville problem \eqref{SL_eq}--\eqref{SL_eq_bc} with isoperimetric problems of discrete fractional calculus of variations.

\begin{theorem}\label{th:VP}
Let $y^1$ denote the first eigenfunction, normalized to satisfy the isoperimetric constraint
\begin{equation}\label{IC}
I[y]=\sum_{k=1}^{N}r(x_{k})(y(x_{k}))^2=1
\end{equation}
associated to the first eigenvalue $\lambda_1$ of problem \eqref{SL_eq}--\eqref{SL_eq_bc}. Then $y^1$ is a minimizer of functional
\begin{equation}\label{VP}
J[y]=\sum_{k=1}^{N}\left[p(x_k)\left({_{0}}\Delta^{\alpha}_{k}y(x_k)\right)^2+q(x_k)(y(x_k))^2\right]
\end{equation}
subject to boundary condition $y(x_0)=0$, $y(x_N)=0$ and isoperimetric constraint \eqref{IC}. Moreover $J[y^1]=\lambda_1$.
\end{theorem}

\begin{proof}
Suppose that $y$ is a minimizer of $J$. Then, by Theorem~5 \cite{malodz}, there exists a real constant $\lambda$ such that $y$ satisfies equation
\begin{equation}\label{eq:SLE0}
{_{k}}\Delta^{\alpha}_{N}\left(p(x_k){_{0}}\Delta^{\alpha}_{k}y(x_k)\right)+q(x_k)y(x_k)-\lambda r(x_k)y(x_k)=0, \quad k=1,\ldots,N-1,
\end{equation}
together with $y(x_0)=0$, $y(x_N)=0$ and isoperimetric constraint \eqref{IC}.
Let us multiply \eqref{eq:SLE0} by $y(x_{k})$ and sum up from $k=1$ to $N-1$, then
\begin{equation*}
\sum_{k=1}^{N-1}\left[y(x_k){_{k}}\Delta^{\alpha}_{N}\left(p(x_k){_{0}}\Delta^{\alpha}_{k}y(x_k)\right)+q(x_k)(y(x_k))^2\right]=\sum_{k=1}^{N-1}\lambda r(x_k)(y(x_k))^2
\end{equation*}
By summation by parts \eqref{int_by_parts}
\begin{equation*}
\sum_{k=1}^{N-1}y(x_k){_{k}}\Delta^{\alpha}_{N}\left(p(x_k){_{0}}\Delta^{\alpha}_{k}y(x_k)\right)=
\sum_{k=1}^{N}p(x_k)\left({_{0}}\Delta^{\alpha}_{k}y(x_k)\right)^2.
\end{equation*}
As \eqref{IC} holds and $y(x_N)=0$ we obtain
\begin{equation*}
J[y]=\lambda.
\end{equation*}
Any solution to problem \eqref{IC}--\eqref{VP} that satisfies equation \eqref{eq:SLE0} must be nontrivial since \eqref{IC} holds, so $\lambda$ must be an eigenvalue. According to Theorem~\ref{existence_SL} there is the least element in the spectrum being eigenvalue $\lambda_1$, and the corresponding eigenfunction $y^{(1)}$ normalized to meet the isoperimetric condition. Therefore $J[y^{(1)}]=\lambda_1$.
\end{proof}

\begin{definition}
We will call functional $R$ defined by
\begin{equation*}
R[y]=\frac{J[y]}{I[y]},
\end{equation*}
where $J[y]$ is given by \eqref{VP} and $I[y]$ by \eqref{IC}, the Rayleigh quotient for the fractional discrete Sturm--Liouville problem \eqref{SL_eq}--\eqref{SL_eq_bc}.
\end{definition}

\begin{theorem}\label{thm:RQ}
Assume that $y$ satisfies boundary conditions $y(x_0)=y(x_N)=0$ and is nontrivial.
 \begin{itemize}
 \item[(i)] If $y$ is a minimizer of Rayleigh quotient $R$ for the Sturm--Liouville problem \eqref{SL_eq}--\eqref{SL_eq_bc}, then value of $R$ in $y$ is equal to the smallest eigenvalue $\lambda_1$, i.e., $R[y]=\lambda_1$.
  \item[(ii)] If $y$ is a maximizer of Rayleigh quotient $R$ for the Sturm--Liouville problem \eqref{SL_eq}--\eqref{SL_eq_bc}, then value of $R$ in $y$ is equal to the largest eigenvalue $\lambda_{N-1}$, i.e., $R[y]=\lambda_{N-1}$.
      \end{itemize}
\end{theorem}

\begin{proof}
We give the proof only for the case (i) as the second case can be proved similarly. Suppose that $y$ satisfying boundary conditions $y(x_0)=y(x_N)=0$ and being nontrivial, is a minimizer of Rayleigh quotient $R$ and that value of $R$ in $y$ is equal to $\lambda$.
Consider the following functions
\begin{equation*}
\fonction{\phi}{[-\varepsilon,\varepsilon]}{\R}{h}{I[y+h\eta]=\displaystyle\sum_{k=1}^{N}r(x_{k})(y(x_{k})+h\eta(x_{k}))^2}
\end{equation*}
\begin{equation*}
\fonction{\psi}{[-\varepsilon,\varepsilon]}{\R}{h}{J[y+h\eta]=\displaystyle \sum_{k=1}^{N}\left[p(x_k)\left({_{0}}\Delta^{\alpha}_{k}(y(x_k)+h\eta(x_{k}))\right)^2+q(x_k)(y(x_k)+h\eta(x_{k}))^2\right]}
\end{equation*}
and
\begin{equation*}
\fonction{\zeta}{[-\varepsilon,\varepsilon]}{\R}{h}{R[y+h\eta]=\frac{J[y+h\eta]}{I[y+h\eta]},}
\end{equation*}
where $\eta: D \rightarrow \R$, $\eta(x_0)=\eta(x_N)=0$, $\eta\neq 0$.
Since $\zeta$ is of class $C^1$ on $[-\varepsilon,\varepsilon]$ and
$$
\zeta(0)\leq\zeta (h),~~\left|h\right|\leq\varepsilon,
$$
we deduce that
\begin{equation*}
\zeta'(0)=\left.\frac{d}{dh}R[y+h\eta]\right|_{h=0}=0.
\end{equation*}
Moreover, notice that
\begin{equation*}
\zeta'(h)=\frac{1}{\phi(h)}\left(\psi'(h)-\frac{\psi(h)}{\phi(h)}\phi'(h)\right)
\end{equation*}
and
\begin{equation*}
\psi'(0)=\left.\frac{d}{dh}J[y+h\eta]\right|_{h=0}=2\sum_{k=1}^{N}\left[p(x_k){_{0}}\Delta^{\alpha}_{k}y(x_k){_{0}}\Delta^{\alpha}_{k}\eta(x_k)
+q(x_k)y(x_k)\eta(x_k)\right],
\end{equation*}
\begin{equation*}
\phi'(0)=\left.\frac{d}{dh}I[y+h\eta]\right|_{h=0}=2\sum_{k=1}^{N}r(x_{k})y(x_{k})\eta(x_{k}).
\end{equation*}
Therefore
\begin{multline*}
\zeta'(0)=\left.\frac{d}{dh}R[y+h\eta]\right|_{h=0}\\
=\frac{2}{I[y]}\left[\sum_{k=1}^{N}\left[p(x_k){_{0}}\Delta^{\alpha}_{k}y(x_k){_{0}}\Delta^{\alpha}_{k}\eta(x_k)
+q(x_k)y(x_k)\eta(x_k)\right]-\frac{J[y]}{I[y]}\sum_{k=1}^{N}r(x_{k})y(x_{k})\eta(x_{k})\right]=0.
\end{multline*}
Having in mind that $\frac{J[y]}{I[y]}=\lambda$, $\eta(x_0)=\eta(x_N)=0$ and using the summation by parts formula \eqref{int_by_parts} we obtain
\begin{equation*}
\sum_{k=1}^{N-1}\left[{_{k}}\Delta^{\alpha}_{N}\left(p(x_k){_{0}}\Delta^{\alpha}_{k}y(x_k)\right)+q(x_k)y(x_k)-\lambda r(x_k)y(x_k)\right]\eta(x_k)=0.
\end{equation*}
Since $\eta$ is arbitrary, we have
\begin{equation}\label{eq:aa}
{_{k}}\Delta^{\alpha}_{N}\left(p(x_k){_{0}}\Delta^{\alpha}_{k}y(x_k)\right)+q(x_k)y(x_k)-\lambda r(x_k)y(x_k)=0,\quad k=1,\ldots,N-1.
\end{equation}
As $y\neq 0$ we have that $\lambda$ is an eigenvalue of \eqref{eq:aa}.
On the other hand, let $\lambda_i$ be an eigenvalue and $y^i$ the corresponding eigenfunction, then
\begin{equation}\label{eq:SLRa}
{_{k}}\Delta^{\alpha}_{N}\left(p(x_k){_{0}}\Delta^{\alpha}_{k}y_i(x_k)\right)+q(x_k)y_i(x_k)=\lambda_i r(x_k)y^i(x_k).
\end{equation}
Similarly to the proof of Theorem~\ref{th:VP}, we can obtain
\begin{equation*}
\frac{\displaystyle \sum_{k=1}^{N}\left[p(x_k)\left({_{0}}\Delta^{\alpha}_{k}y^i(x_k)\right)^2+q(x_k)(y^i(x_k))^2\right]}{\displaystyle \sum_{k=1}^{N}r(x_{k})(y^i(x_{k}))^2}=\lambda_i,
\end{equation*}
for any $1 \leq i \leq N-1$. That is $R[y^i]=\frac{J[y^i]}{I[y^{i}]}=\lambda_{i}$. Finally, since the minimum value of $R$ at $y$ is equal to $\lambda$, i.e.,
\begin{equation*}
\lambda\leq R[y^{i}]=\lambda_i~~\forall i\in\{1,\ldots,N-1\}
\end{equation*}
we have $\lambda=\lambda_1$.
\end{proof}

\begin{example}
Let us consider the following problem: minimize
\begin{equation}\label{ex:f}
J[y]=\sum_{k=1}^{N}\left({_{0}}\Delta^{\alpha}_{k}y(x_k)\right)^2
\end{equation}
subject to
\begin{equation}\label{ex:f:3}
I[y]=\sum_{k=1}^{N}\left(y(x_k)\right)^2=1
\end{equation}
and $y(x_0)=y(x_N)=0$, where  $N$ is fixed. In this case the Euler--Lagrange equation takes the form
\begin{equation}\label{ex:SL}
{_{k}}\Delta^{\alpha}_{N}{_{0}}\Delta^{\alpha}_{k}y(x_k)=\lambda y(x_k), \quad k=1,\ldots,N-1.
\end{equation}
Together with boundary condition $y(x_0)=y(x_N)=0$ it is the Sturm--Liouville eigenvalue problem where
$p(x_i)=1$, $r(x_i)=1$ and $q(x_i)=0$ for $k=1,\ldots,N-1$. Let us choose $N=4$ and $h=1$. Eigenvalues of \eqref{ex:SL} for different values of $\alpha$'s are presented in Table~\ref{Tb1}. Those results are obtained by solving the matrix eigenvalue problem of the form \eqref{M2_sl}.

\begin{table}[ht]
\centering
\begin{tabular}{|c|c|c|c|c|}\hline

$\alpha$ & $\lambda_1$ & $ \lambda_2$ &$\lambda_3$ \\ \hline

$0.25$  & $0.7102065750$ & $1.148567387$ &$1.349294886$  \\ \hline

$0.50$  &$0.6004483933$ & $1.353660384$ &$1.831047473$  \\ \hline

$0.75$  & $0.5779798778$ & $1.632135974$ &$2.496488153$  \\ \hline

$1$  & $0.5857864376$ & $2.0$ &$3.414213562$  \\ \hline
\end{tabular}
\\
\caption{Eigenvalues of \eqref{ex:SL} for different values of $\alpha$'s: $1/4,1/2,3/4,1$.}
\label{Tb1}
\end{table}

 Observe that problem \eqref{ex:f}--\eqref{ex:f:3} can be treated as a finite dimensional constrained optimization problem. Namely, the problem is to minimize function $J$ of $N-1$ variables: $y_1=y(x_1),\ldots,y_{N-1}=y(x_{N-1})$ on the $N-1$ dimensional sphere with equation $\sum_{k=1}^{N-1}y_k^2=1$. Table~\ref{Tb} and Figure~\ref{plotdys} present the solution to problem \eqref{ex:f}--\eqref{ex:f:3} for $N=4$, $h=1$ and different values of $\alpha$'s. By Theorem~\ref{th:VP} the first eigenvalue $\lambda_1$ of \eqref{ex:SL} is the minimum value of $J$ on $\sum_{k=1}^{N_1}y_k^2=1$ and the first eigenfunction of \eqref{ex:SL} is the minimizer of this problem.  Other eigenfunctions and eigenvalues of  \eqref{ex:SL} we can found by using the first order necessary optimality conditions (Karush--Kuhn--Tucker conditions), that is, by solving the following system of equations:
 \begin{equation}\label{system}
\left\{
  \begin{array}{l}
   \frac{\partial J}{\partial y_k}=\lambda \frac{\partial I}{\partial y_k}, \quad k=1,\ldots,N-1,\\
    \sum_{k=1}^{N-1}y_k^2=1.
  \end{array}
\right.
\end{equation}

\begin{figure}
  \centering
 \includegraphics[width=7cm]{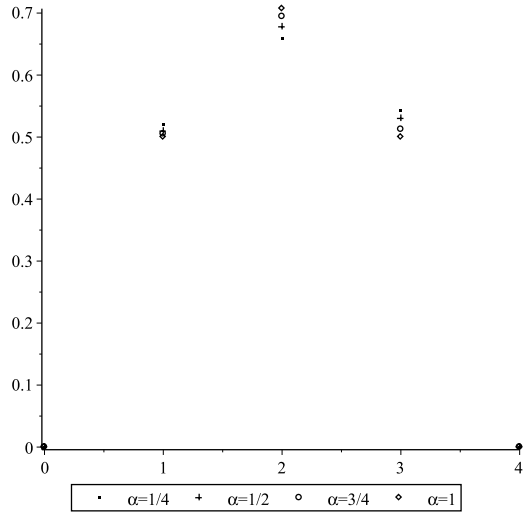}\\
  \caption{The solution to problem \eqref{ex:f}--\eqref{ex:f:3} for different values of $\alpha$'s: $1/4,1/2,3/4,1$.}\label{plotdys}
\end{figure}

\begin{table}[ht]
\centering
\begin{tabular}{|c|c|c|c|c|c|c|}\hline

$\alpha$ & $y(x_1)$ & $ y(x_2)$ & $ y(x_3)$ & $\lambda_1$ \\ \hline

$0.25$  & $0.52042378274$ & $0.65949734450$ &$0.54242265711$ &$0.7102065749$ \\ \hline

$0.50$  &$0.50954825567$ & $0.67778735991$ &$0.53006119446$ &$0.6004483933$ \\ \hline

$0.75$  & $0.50509466979$ & $0.69443334582$ &$0.51248580736$  &$0.5779798777$ \\ \hline

$1$  & $0.49999999999$ & $0.70710678118$ &$0.5$ &$0.5857864376$  \\ \hline
\end{tabular}
\\
\caption{The solution to problem \eqref{ex:f}--\eqref{ex:f:3} for different values of $\alpha$'s: $1/4,1/2,3/4,1$.}
\label{Tb}
\end{table}

\end{example}

\section{Continuous Fractional Calculus}
\label{sec:CFC}

This section is devoted to the continuous fractional Sturm--Liouville problem and its formulation as a constrained fractional variational principle. Namely, we shall show that this formulation can be used to approximate the solutions. As in the discrete case there are several different definitions for fractional derivatives \cite{Kilbas}, the most well known are the Gr\"{u}nwald--Letnikov, the Riemann--Liouville and the Caputo fractional derivatives.
\begin{definition}
Let $f:[a,b]\rightarrow\mathbb{R}$ be a function and  $\alpha$ a positive real number such that $0<\alpha<1$. We define
\begin{enumerate}
\item{the left and right Riemann--Liouville fractional derivatives of order $\alpha$ by
$${_aD_x^\alpha}f(x):=\frac{1}{\Gamma(1-\alpha)}\frac{d}{dx}\int_a^x(x-t)^{-\alpha}f(t)dt,$$
and
$${_xD_b^\alpha}f(x):=\frac{-1}{\Gamma(1-\alpha)}\frac{d}{dx}\int_x^b(t-x)^{-\alpha} f(t)dt,$$
respectively;}
\item{the left and right Caputo fractional derivatives of order $\alpha$ by
$${_a^CD_x^\alpha}f(x):=\frac{1}{\Gamma(1-\alpha)}\int_a^x (x-t)^{-\alpha}f'(t)dt,$$
and
$${_x^CD_b^\alpha}f(x):=\frac{-1}{\Gamma(1-\alpha)}\int_x^b(t-x)^{-\alpha} f'(t)dt,$$
respectively.}
\end{enumerate}
\end{definition}

The Caputo derivative seems more suitable in applications. Let us recall that the Caputo
derivative of a constant is zero, whereas for the Riemann--Liouville is not. Moreover, the Laplace transform, which is used for solving fractional differential equations, of the Riemann-Liouville derivative contains the limit values of
the Riemann-Liouville fractional derivatives (of order $\alpha-1$) at the lower
terminal $x = a$. Mathematically such problems can be solved, but there is no physical interpretation
for such type of conditions. On the other hand the Laplace transform of the Caputo derivative imposes
boundary conditions involving the value of function at the lower point $x = a$ which usually are acceptable physical conditions. \\
The Gr\"{u}nwald--Letnikov definition is a generalization of the ordinary discretization formulas for integer order derivatives.

\begin{definition}
Let $0<\alpha<1$ be a real. The left and right Gr\"{u}nwald--Letnikov fractional derivative of a function $f$, of order $\alpha$, is defined as
$${_a^{GL}D_x^\alpha}f(x):=\lim_{h\rightarrow 0^+} \frac{1}{h^\alpha}\sum_{k=0}^\infty(-1)^k\binom{\alpha}{k}f(x-kh),$$
and
$${_x^{GL}D_b^\alpha}f(x):=\lim_{h\rightarrow 0^+} \frac{1}{h^\alpha}\sum_{k=0}^\infty(-1)^k\binom{\alpha}{k}f(x+kh),$$
respectively.
\end{definition}

Here $\binom{\alpha}{k}$ stands for the generalization of binomial coefficients to real numbers (see \eqref{coef}). However in this section, for historical reasons, we denote
$$(w_k^\alpha):= (-1)^k\binom{\alpha}{k}$$
rather than $a_i^{(\alpha)}$.

Relations between those three types of derivatives are given below and can be
found respectively in \cite{book:Podlubny,Kilbas}.

\begin{proposition}
Let us assume that the function $f$ is integrable in $[a, b]$. Then,
the Riemann-Liouville fractional derivatives exist and coincide with Gr\"{u}nwald--Letnikov fractional derivatives.
\end{proposition}

\begin{proposition}
Let us assume that $f$ is
a function for which the Caputo fractional derivatives
exist together with the Riemann-Liouville fractional
derivatives in $[a, b]$. Then
\begin{equation}\label{approx_Caputo}{_a^CD_x^\alpha}f(x)={_aD_x^\alpha}f(x)-\frac{f(a)}{\Gamma(1-\alpha)}(x-a)^{-\alpha}\end{equation}
and
$${_x^CD_b^\alpha}f(x)={_xD_b^\alpha}f(x)-\frac{f(b)}{\Gamma(1-\alpha)}(b-x)^{-\alpha}.$$
If $f(a)=0$ or $ f(b)=0$, then ${_a^CD_x^\alpha}f(x)={_aD_x^\alpha}f(x)$ or ${_x^CD_b^\alpha}f(x)={_xD_b^\alpha}f(x)$, respectively.
\end{proposition}

It is well known that we can approximate the Riemann--Liouville fractional derivative using the Gr\"{u}nwald--Letnikov fractional derivative.
Given the interval $[a,b]$ and a partition of the interval $x_j=a+ jh$, for $j = 0,1,...,N$ and some $h>0$ such that $x_N=b$, we have
$${_aD_{x_j}^\alpha}f(x_j)= \frac{1}{h^\alpha}\sum_{k=0}^j(w_k^\alpha) f(x_{j-k})+O(h),$$
$${_{x_j}D_{b}^\alpha}f(x_j)= \frac{1}{h^\alpha}\sum_{k=0}^{N-j}(w_k^\alpha) f(x_{j+k})+O(h),$$
that is, the truncated Gr\"{u}nwald--Letnikov fractional derivatives are first-order approximations of the Riemann--Liouville fractional derivatives. Using the relation \eqref{approx_Caputo}, we deduce a decomposition sum for the Caputo fractional derivatives:
\begin{equation}\label{Caputo_appro:1}
{_a^CD_{x_j}^\alpha}f(x_j)\approx \frac{1}{h^\alpha}\sum_{k=0}^j(w_k^\alpha) f(x_{j-k})-\frac{f(a)}{\Gamma(1-\alpha)}(x_j-a)^{-\alpha}=: {_a^C\tilde{D}_{x_j}^\alpha}f(x_j),\end{equation}
\begin{equation}\label{Caputo_appro:2}
{_{x_j}^CD_{b}^\alpha}f(x_j)\approx \frac{1}{h^\alpha}\sum_{k=0}^{N-j}(w_k^\alpha) f(x_{j+k})-\frac{f(b)}{\Gamma(1-\alpha)}(b-x_j)^{-\alpha}=: {_{x_j}^C\tilde{D}_{b}^\alpha}f(x_j).\end{equation}

\subsection{Variational Problem}
\label{CVPProblem2}

Consider the following variational problem: to minimize the functional
\begin{equation}\label{C_functional}I[y]=\int_a^b L(x,y(x),{_a^CD_x^\alpha}y(x))\,dx,\end{equation}
subject to the boundary conditions
\begin{equation}\label{C_boundary} y(a)=y_a \, \mbox{ and } \, y(b)=y_b, \quad y_a,y_b\in\mathbb R,\end{equation}
where $0<\alpha<1$ and the Lagrange function $L:[a,b]\times \mathbb R^2\rightarrow\mathbb{R}$ is differentiable with respect to the second and third arguments.

\begin{theorem}(\cite{Agrawal1})
If $\overline y$ is a solution to \eqref{C_functional}--\eqref{C_boundary}, then $\overline y$ satisfies the following fractional differential equation
\begin{equation}\label{EQ_EL}\frac{\partial L}{\partial y}(x,y(x),{_a^CD_x^\alpha}y(x))+{_xD_b^\alpha}\frac{\partial L}{\partial{_a^CD_x^\alpha} y}(x,y(x),{_a^CD_x^\alpha}y(x))=0, \quad t\in[a,b].\end{equation}
\end{theorem}

Relations like \eqref{EQ_EL} are known in the literature as the Euler--Lagrange equation, and provide a necessary condition that every solution of the variational problem must verify.
Adding to problem \eqref{C_functional}--\eqref{C_boundary} an integral constraint
\begin{equation}\label{C_constraint}\int_a^b g(x,y(x),{_a^CD_x^\alpha}y(x))\,dx=K,\end{equation}
where $K$ is a fixed constant and $g:[a,b]\times \mathbb R^2\rightarrow\mathbb{R}$ is a differentiable function with respect to the second and third arguments, we get an isoperimetric variational problem. In order to obtain a necessary condition for a minimizer we define the new function
\begin{equation}\label{C_iso}F:=\lambda_0 L(x,y(x),{_a^CD_x^\alpha}y(x))-\lambda g(x,y(x),{_a^CD_x^\alpha}y(x)),\end{equation}
where $\lambda_0,\lambda$ are Lagrange multipliers. Then every solution $\overline y$ of the fractional isoperimetric problem given by \eqref{C_functional}--\eqref{C_boundary} and \eqref{C_constraint} is also a solution to the fractional differential equation (c.f. \cite{Almeida1})
\begin{equation}\label{EQ_EL2}\frac{\partial F}{\partial y}(x,y(x),{_a^CD_x^\alpha}y(x))+{_xD_b^\alpha}\frac{\partial F}{\partial{_a^CD_x^\alpha} y}(x,y(x),{_a^CD_x^\alpha}y(x))=0, \quad t\in[a,b].\end{equation}
Moreover, if $\overline y$ is not a solution to
\begin{equation}\label{EL_I}
\frac{\partial g}{\partial y}(x,y(x),{_a^CD_x^\alpha}y(x))+{_xD_b^\alpha}\frac{\partial g}{\partial{_a^CD_x^\alpha} y}(x,y(x),{_a^CD_x^\alpha}y(x))=0, \quad t\in[a,b],
\end{equation}
then we can put $\lambda_0=1$ in \eqref{C_iso}.

\subsubsection{Discretization Method 1}
\label{Discretization1}

 Using the approximation formula for the Caputo fractional derivative given by \eqref{Caputo_appro:1}, we can discretize functional \eqref{C_functional} in the following way. Let $N\in\mathbb N$, $h=(b-a)/N$ and the grid $x_j=a+jh$, $j=0,1,\ldots, N$. Then
\begin{equation}\label{Discr1}\begin{array}{ll}
I[y]&=\displaystyle \sum_{k=1}^N\int_{x_{k-1}}^{x_k} L(x,y(x),{_a^CD_x^\alpha}y(x))\,dx\\
&\approx \displaystyle \sum_{k=1}^N h L(x_k,y(x_k),{_a^CD_{x_k}^\alpha}y(x_k))\\
&\approx \displaystyle \sum_{k=1}^N h L(x_k,y(x_k),{_a^C\tilde{D}_{x_k}^\alpha}y(x_k)).\\
\end{array}\end{equation}
This is the direct way to solve the problem, using discretization techniques.

\subsubsection{Discretization Method 2}
\label{Discretization2}

By the previous discussion, the initial problem of minimization of the functional  \eqref{C_functional}, subject to boundary conditions \eqref{C_boundary}, can be numerically replaced by the finite dimensional optimization problem
$$ \Phi(y_1,\ldots,y_{N-1}):= \sum_{k=1}^N h L(x_k,y(x_k),{_a^C\tilde{D}_{x_k}^\alpha}y(x_k)) \, \rightarrow \, \min,$$
subject to
$$y_0=y_a \, \mbox{ and } \, y_N=y_b$$
where $y_k:=y(x_k)$.

%
Using the first order necessary optimality conditions given by the following system of $N-1$ equations:
$$\frac{\partial \Phi}{\partial y_j}=0, \quad \forall j=1,\ldots, N-1,$$
we get
\begin{equation}\label{system1}\frac{\partial L}{\partial y}(x_j,y(x_j),{_a^C\tilde{D}_{x_j}^\alpha}y(x_j))+\sum_{k=0}^{N-j}\frac{(w_k^\alpha)}{h^\alpha}
\frac{\partial L}{\partial _a^CD_{x}^\alpha y}(x_{j+k},y(x_{j+k}),{_a^C\tilde{D}_{x_{j+k}}^\alpha}y(x_{j+k}))=0\end{equation}
$j=1,\ldots, N-1$.
As $N\to\infty$, that is, as $h\to0$, the solutions of system \eqref{system1} converge to the solutions of the fractional Euler--Lagrange equation associated to the variational problem (see \cite[Theorem 4.1]{Pooseh1}). Constrained variational problem given by \eqref{C_functional}--\eqref{C_boundary} and \eqref{C_constraint} can be solved similarly. More precisely, in this case we have to replace the Lagrange function $L$ by the augmented function $F=\lambda_0L-\lambda g$, and proceed with similar calculations.


\subsection{Sturm--Liouville problem}
\label{Sturm--Liouville problem}

Consider the fractional differential equation
\begin{equation}\label{eq:SLE}
\left[{}^{C}D^{\alpha}_{b}p(x){}^{C}D^{\alpha}_{a}+q(x)\right]y(x)= \lambda r_{\alpha}(x)y(x),
\end{equation}
subject to the boundary conditions
\begin{equation}\label{eq:BC}
y(a)=y(b)=0.
\end{equation}
Equation \eqref{eq:SLE}  together with condition \eqref{eq:BC} is called the fractional Sturm--Liouville problem. As in the discrete case, it is required to find the eigenfunctions and the eigenvalues of the given boundary value problem, i.e., the nontrivial solutions of \eqref{eq:SLE}--\eqref{eq:BC} and the corresponding values of the parameter $\lambda$.

In what follows we assume:
\begin{description}
\item[$(A)$] Let $\frac{1}{2}<\alpha<1$ and $p,q,r_{\alpha}$ be given functions such that: $p\in C^1[a,b]$
and $p(x)>0$ for all $x\in [a,b]$; $q,r_{\alpha}\in C[a,b]$, $r_{\alpha}(x)>0$ for all $x\in [a,b]$ and $(\sqrt{r_{\alpha}})'$ is H\"{o}lderian, of order $\beta\leq \alpha - \frac{1}{2}$, on $[a,b]$.
\end{description}

\begin{theorem}\label{thm:exist}(\cite{Tatiana1})
Under assumption (A), the fractional Sturm--Liouville Problem \eqref{eq:SLE}--\eqref{eq:BC}
has an infinite increasing sequence of eigenvalues $\lambda_{1}, \lambda_{2},...,$
and to each eigenvalue $\lambda_{k}$ there is a corresponding continuous eigenfunction
$y_{k}$ which is unique up to a constant factor.
\end{theorem}

The fractional Sturm--Liouville problem can be remodeled as a fractional isoperimetric variational problem.

\begin{theorem}\label{thm:FE}(\cite{Tatiana1})
Let assumption (A) holds and $y^{1}$ be the eigenfunction, normalized to satisfy the isoperimetric constraint
\begin{equation}\label{eq:IC2}
I[y]=\int_a^{b}r_{\alpha}(x)y^2(x)\;dx=1,
\end{equation}
associated to the first eigenvalue $\lambda_{1}$ of problem \eqref{eq:SLE}--\eqref{eq:BC} and assume that function $D^{\alpha}_{b}(p{}^{C}D^{\alpha}_{a}y^{1})$ is continuous. Then, $y^{1}$ is a minimizer of the following variational functional:
\begin{equation}\label{eq:F2}
J[y]=\int_a^{b}\left[p(x) ({}^{C}D^{\alpha}_{a}y(x))^{2}+q(x)y^{2}(x)\right]\;dx ,
\end{equation}
in the class of $C[a,b]$ functions with ${}^{C}D^{\alpha}_{a}y$  and $D^{\alpha}_{b}(p{}^{C}D^{\alpha}_{a}y)$ continuous in $[a,b]$, subject to the boundary conditions
\begin{equation}\label{eq:BC2}
y(a)=y(b)=0
\end{equation}
and isoperimetric constraint \eqref{eq:IC2}. Moreover,
\begin{equation*}
J[y^{1}]=\lambda_{1}.
\end{equation*}
\end{theorem}

\subsubsection{Discretization Method 3}
\label{Discretization1:SL}

 Using the approximation formula for the Caputo fractional derivatives given by \eqref{Caputo_appro:1}--\eqref{Caputo_appro:2}, we can discretize equation \eqref{eq:SLE} in the following way. Let $N\in\mathbb N$, $h=(b-a)/N$ and the grid $x_j=a+jh$, $j=0,1,\ldots, N$. Then at $x=x_i$, \eqref{eq:SLE} may be discretized as:

\[\frac{h^{-2\alpha}}{r_{\alpha}(x_i)}\sum_{k=0}^{N-i}\left(w_k^{\alpha}\right)p(x_{i+k})\sum_{l=0}^{i+k}\left(w_l^{\alpha}\right)y_{i+k-l}
+\frac{q(x_i)}{r_{\alpha}(x_i)}y_i=\lambda y_i, \quad i=1,\ldots,N-1,\]
which in the matrix form, may be written as
\begin{equation}\label{MF_GL}AY=\lambda Y,\end{equation}
where $Y=[y_1~~y_2\ldots y_{N-1}]$, $y_i=y(x_i)$, and $A=\left(c_{ik}\right)$, $i=1,2,\ldots,N-1$, $k=1,2,\ldots,N-1$, with $\displaystyle c_{ik}=\begin{cases}
\frac{h^{-2\alpha}}{r_{\alpha}(x_i)}\sum_{j=0}^{N-i}\left(w_j^{\alpha}\right)^2p(x_{j+i})+\frac{q(x_i)}{r_{\alpha}(x_i)}, & i=k\\
\frac{h^{-2\alpha}}{r_{\alpha}(x_i)}\sum_{j=0}^{N-i}\left(w_j^{\alpha}\right)\left(w_{j+i-k}^{\alpha}\right)p(x_{j+i}), & i>k\\
\frac{h^{-2\alpha}}{r_{\alpha}(x_i)}\sum_{j=k-i}^{N-i}\left(w_j^{\alpha}\right)\left(w_{j+i-k}^{\alpha}\right)p(x_{j+i}), & i<k
\end{cases}$,\\
reducing in this way the Sturm-Liouville problem to an algebraic eigenvalue problem.

%
%

\begin{example}
Let us consider the following problem:
minimize the functional
\begin{equation}\label{ex:e:2}
\int_0^1({_0^CD_x^\alpha}y(x))^2\,dx,
\end{equation}
under the restrictions
\begin{equation}\label{ex:e:22}
\int_0^1y^2(x)\,dx=1 \quad \mbox{and} \quad y(0)=y(1)=0,
\end{equation}
where $\alpha=3/4$.
Since $y(0)=0$, we have ${_0^CD_x^\alpha}y(x)={_0D_x^\alpha}y(x)$. Discretizing the problem, as explained in Section~\ref{Discretization1}, we obtain a finite dimensional constrained optimization problem:
\begin{equation}\label{op:1}
\sum_{k=1}^N N^{2\alpha-1}\left(\sum_{i=0}^k (w^\alpha_i)y_{k-i}\right)^2 \, \rightarrow \, \min,
\end{equation}
subject to
\begin{equation}\label{op:2}
\sum_{k=1}^N\frac{y^2_k}{N}=1  \quad \mbox{and} \quad  y_0=y_N=0.
\end{equation}
Using the Maple package \textit{Optimization}, we get approximations of the optimal solutions to \eqref{ex:e:2}--\eqref{ex:e:22} for different values of $N$. Table~\ref{table1} shows values of $\lambda_1$ for $N=5,10,15$. Note that $\lambda_1$ is the value of \eqref{op:1}, where $\overline y=[0,y_1,\ldots,y_{N-1},0]$ is the optimal solution to \eqref{op:1}--\eqref{op:2}. In other words, $\lambda_1$ is an approximation of the minimum value of functional \eqref{ex:e:2} and the first eigenvalue of the Sturm--Liouville (which is the Euler--Lagrange equation for considered variational problem). Figure~\ref{plot1} presents minimizers $\overline y$ for $N=5,10,15$.

\begin{table}[ht]
\centering
\begin{tabular}{|c|c|c|c|}\hline
 $N$ & $5$ & $10$ & $15$\\ \hline
 $\lambda_1$ & $4.603751971$ & $4.491185175$&$4.426964914$ \\ \hline
\end{tabular}
\\
\caption{Values of $\lambda_1$ for $N=5,10,15$.}
\label{table1}
\end{table}

\begin{figure}
  \centering
  \includegraphics[width=7cm]{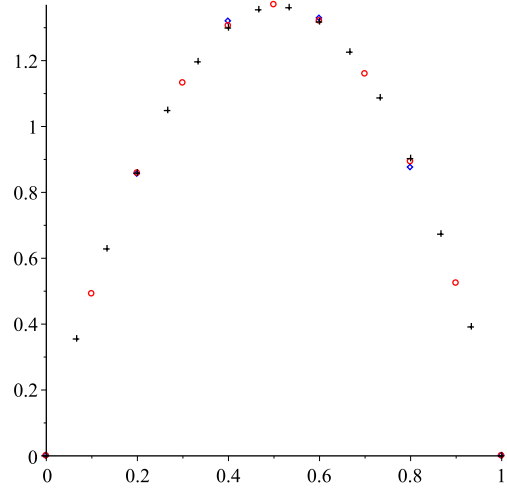}\\
  \caption{Approximation of solutions to problem \eqref{ex:e:2}--\eqref{ex:e:22} (Method 1): $\diamond ( N=5)$; $\circ ( N=10)$; $+ (N=15)$.}\label{plot1}
\end{figure}

Observe that the unique solution to the Euler--Lagrange equation (cf. \eqref{EL_I}) associated to the integral constraint is $\overline y (x)=0$. As $\overline y (x)=0$ is not a solution to \eqref{ex:e:2}--\eqref{ex:e:22} (condition $\int_0^1y^2(x)\,dx=1$ fails), we can consider $\lambda_0=1$ in \eqref{C_iso}. Therefore the auxiliary function is
$$F:=({_0^CD_x^\alpha}y(x))^2-\lambda y^2(x).$$
Thus,
$$\Phi(y_1,\ldots,y_{N-1}):=\sum_{k=1}^N h\left(({_0^CD_{x_k}^\alpha}y_k)^2-\lambda y^2_k\right),$$
and the computation of $\partial \Phi / \partial y_j$ leads to
\begin{equation}\label{system}
-\lambda y_j + N^{2\alpha}\sum_{k=0}^{N-j} (w^\alpha_k) \sum_{l=0}^{j+k} (w^\alpha_l)y_{j+k-l}=0, \quad j=1,\ldots, N-1.
\end{equation}
Solving system of equations \eqref{system} together with \eqref{op:2} we obtain not only an approximation of the optimal solution to problem \eqref{ex:e:2}--\eqref{ex:e:22}, but also other solutions to the Euler--Lagrange equation \eqref{EQ_EL2} as $N\to\infty$. In other words, we get some approximations of the eigenvalues and eigenfunctions of the Sturm--Liouville problem. Table~\ref{table2} presents approximations of the eigenvalues obtained by this procedure for $N=5,10,15$.
In Figure \ref{plot2} we present the optimal solution for this procedure, that corresponds to the eigenvector associated with the eigenvalue $\lambda_1$.

\begin{table}[ht]
\centering
\begin{tabular}{|c|c|c|c|}\hline
$N$& $5$ & $10$ & $15$\\ \hline
$\lambda_1$ & $4.603751969$ & $4.491185168$&$4.426964909$ \\ \hline
$\lambda_2$ &$13.67144835$ &$14.31569449$&$14.33350940$\\ \hline
$\lambda_3$ &$22.69092491$& $26.35335634$&$26.90113751$\\ \hline
$\lambda_4$ & $29.24531071$&$39.48118456$ &$41.37391615$\\ \hline
$\lambda_5 $&-- & $52.54234156$ &$56.93534748$ \\ \hline
$\lambda_6 $&-- & $64.64953668$&$73.03700902$\\ \hline
$\lambda_7 $&-- &$74.96494602$ &$89.07875858$\\ \hline
$\lambda_8 $&-- &$82.83813371$  &$104.5749014$\\ \hline
$\lambda_9 $&-- &$87.76536891$  &$119.0339408$\\ \hline
$\lambda_{10}$ &-- & -- &$132.0436041$ \\ \hline
$\lambda_{11} $&-- &--  &$143.2212682$ \\ \hline
$\lambda_{12} $&-- & -- & $152.2566950$\\ \hline
$\lambda_{13} $&-- & -- &$158.8942685$\\ \hline
$\lambda_{14} $&-- & -- &$162.9518168$\\ \hline
\end{tabular}
\\
\caption{Values of $\lambda_{i}$ for $N=5,10,15$.}
\label{table2}
\end{table}

\begin{figure}
  \centering
  \includegraphics[width=7cm]{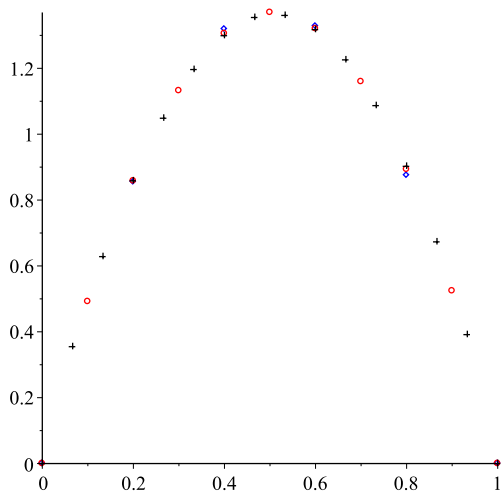}\\
  \caption{Approximation of solutions to problem \eqref{ex:e:2}--\eqref{ex:e:22} (Method 2): $\diamond ( N=5)$; $\circ ( N=10)$; $+ (N=15)$.}\label{plot2}
\end{figure}

In Figures \ref{plot3}, \ref{plot4} and \ref{plot5} we compare the approximation of the optimal solutions to \eqref{ex:e:2}--\eqref{ex:e:22}, obtained by solving \eqref{op:1}--\eqref{op:2} (Method 1) and \eqref{system}--\eqref{op:2} (Method 2), for $N=5,10,15$.

\begin{figure}
  \centering
  \includegraphics[width=7cm]{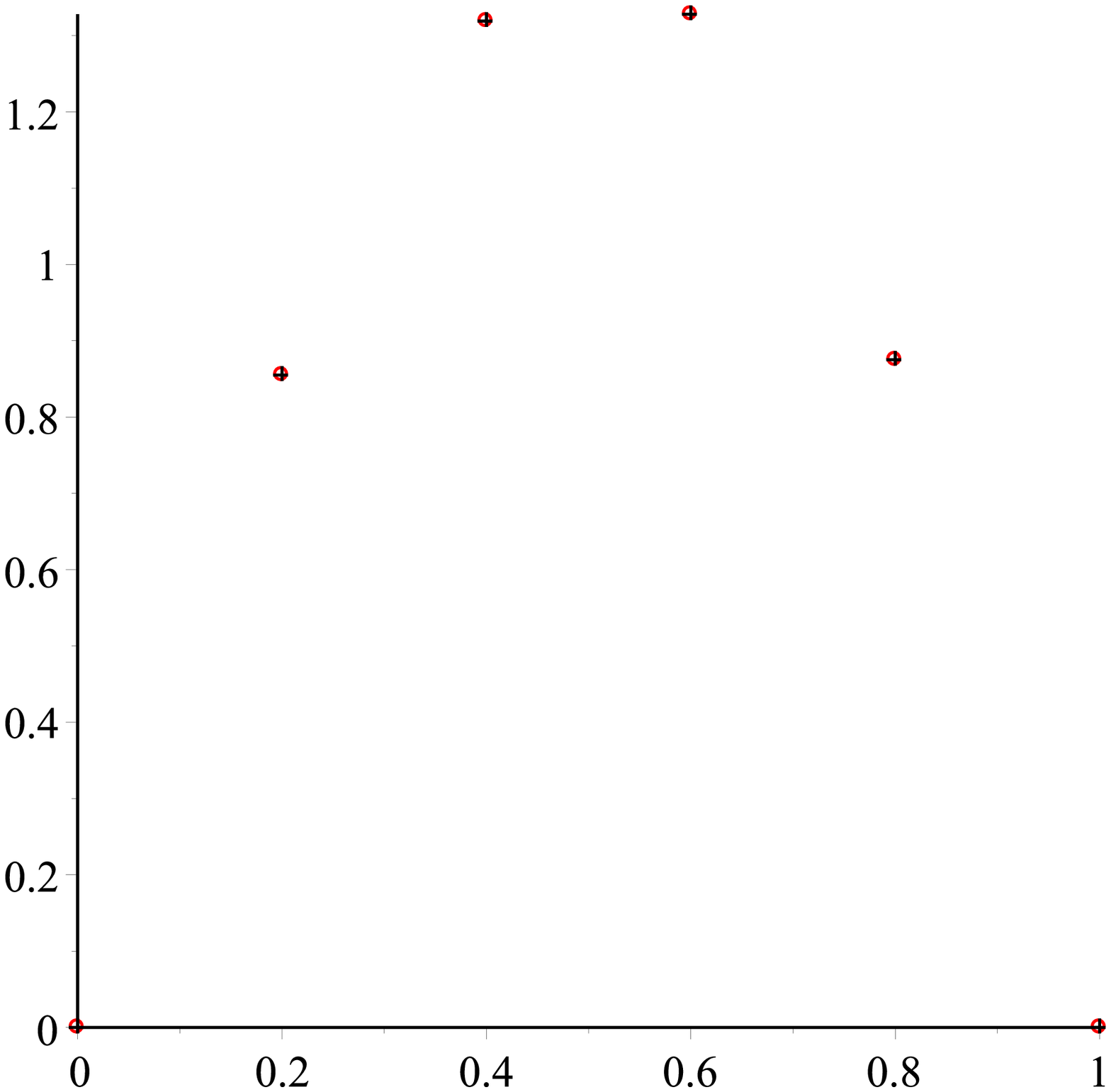}\\
  \caption{$N=5$: $\circ (\mbox{Method 1})$; $+ (\mbox{Method 2})$.}\label{plot3}
\end{figure}

\begin{figure}
  \centering
  \includegraphics[width=7cm]{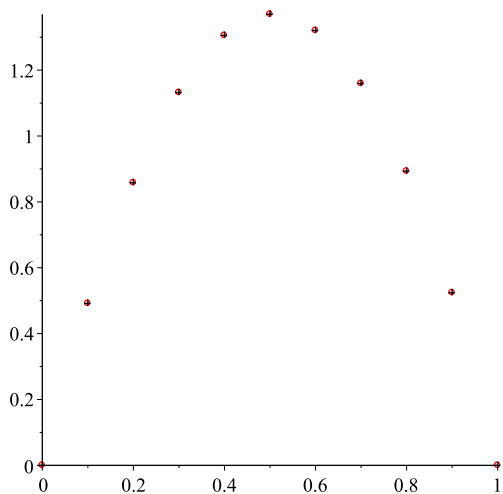}\\
  \caption{$N=10$: $\circ (\mbox{Method 1})$; $+ (\mbox{Method 2})$.}\label{plot4}
\end{figure}

\begin{figure}
  \centering
  \includegraphics[width=7cm]{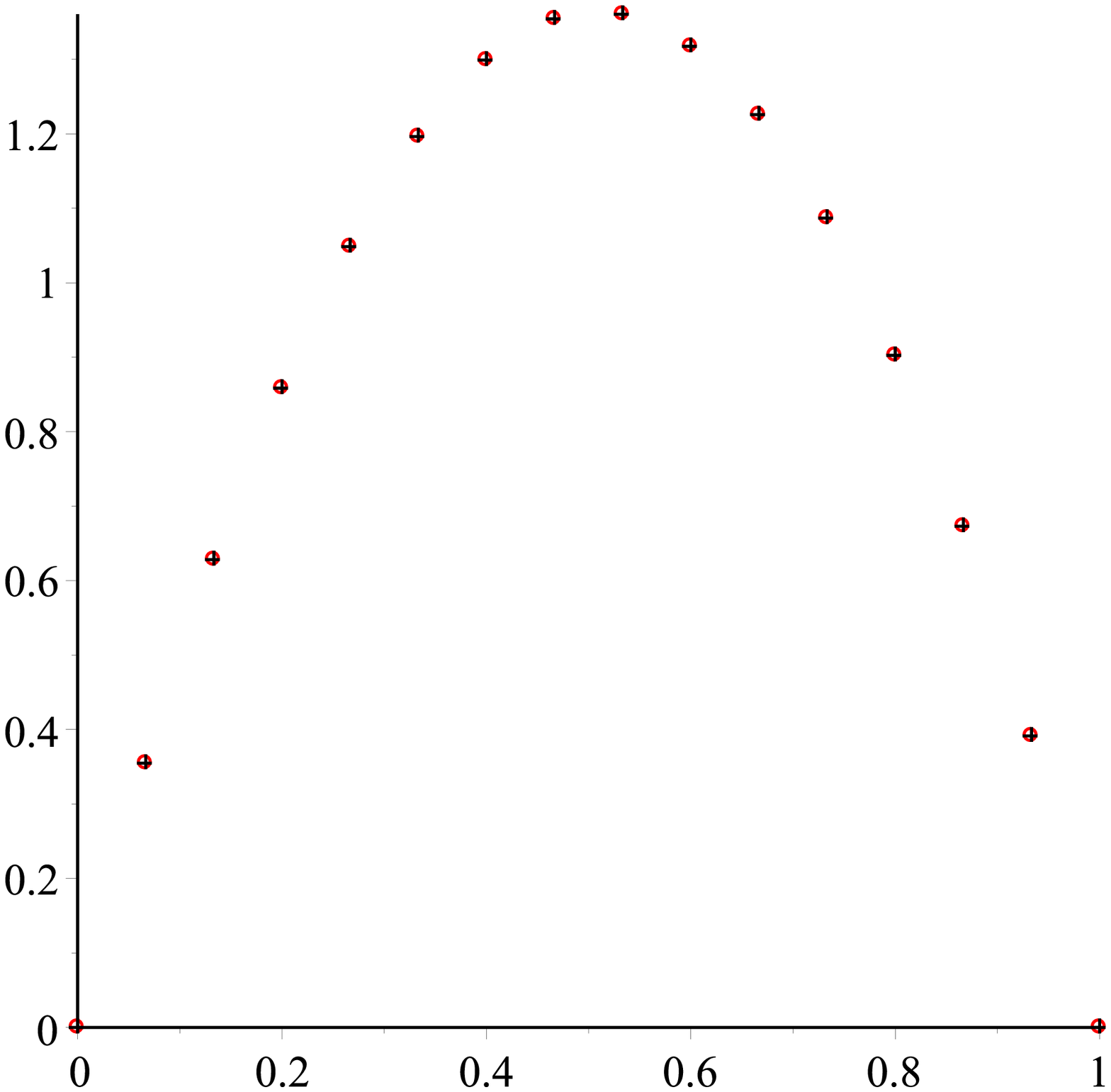}\\
  \caption{$N=15$: $\circ (\mbox{Method 1})$; $+ (\mbox{Method 2})$.}\label{plot5}
\end{figure}
Now let us consider the Sturm--Liouville problem
\begin{equation}\label{eq:ex:2}
{}^{C}D^{3/4}_{1}{}^{C}D^{3/4}_{0}y(x)=\lambda y(x),
\end{equation}
subject to the boundary conditions
\begin{equation}\label{eq:ex:BC}
y(0)=y(1)=0.
\end{equation}
Note that, under conditions of Theorem~\ref{thm:FE}, equation \eqref{eq:ex:2} is the Euler--Lagrange equation for isoperimetric problem \eqref{ex:e:2}--\eqref{ex:e:22}. Table~\ref{table3} presents approximations of the eigenvalues of \eqref{eq:ex:2} obtained by discretization method \ref{Discretization1:SL} for $N=5,10,20,40,80,160$ (for $N=20,40,80,160$ only the first 14 eigenvalues are listed). In Figure~\ref{ef_luisa} we present normalized eigenfunctions, obtained for $N=100$, corresponding to the eigenvalues $\lambda_1$, $\lambda_2$, $\lambda_3$ and $\lambda_4$.

\begin{table}[ht]\label{table3}
\centering
\begin{tabular}{|c|c|c|c|c|c|c|}
\hline
$N$ & $5$ & $10$ & $20$ & $40$ & $80$ & $160$\\
\hline
$\lambda_{1}$ & $4.603751972$ &  $4.491185175$ & $4.387575384$ & $4.314056432$ & $4.264767769$ & $4.231946921$\\
\hline
$\lambda_{2}$ & $13.67144835$ & $14.31569450$ & $14.29943076$ & $14.18275912$ & $14.08194289$ & $14.01015799$ \\
\hline
$\lambda_{3}$ & $22.69092491$ & $26.35335634$ & $27.02784640$ & $27.01132309$ & $26.88877847$ & $26.78184511$ \\
\hline
$\lambda_{4}$ & $29.24531071$ & $39.48118456$ & $41.95747334$ & $42.33045300$ & $42.25841874$ & $42.13429128$ \\
\hline
$\lambda_{5}$ & $-$ & $52.54234157$ & $58.40981791$ & $59.60122278$ & $59.68496380$ & $59.56753673$\\
\hline
$\lambda_{6}$ & $-$ & $64.64953668$ & $75.99486098$ & $78.61012095$ & $79.00578911$ & $78.93437596$\\
\hline
$\lambda_{7}$ & $-$ & $74.96494602$ & $94.25189512$ & $99.05856280$ & $99.96479334$ & $99.98764503$\\
\hline
$\lambda_{8}$ & $-$ & $82.83813372$ & $112.8375161$ & $120.7904806$ & $122.4632696$ & $122.6454529$\\
\hline
$\lambda_{9}$ & $-$ & $87.76536891$ & $131.3694072$ & $143.5891552$ & $146.3337902$ & $146.7516690$\\
\hline
$\lambda_{10}$ & $-$ & $-$ & $149.5318910$ & $167.3194776$ & $171.5033012$ & $172.2513810$\\
\hline
$\lambda_{11}$ & $-$ & $-$ & $166.9946039$ & $191.8029693$ & $197.8472756$ & $199.0332700$\\
\hline
$\lambda_{12}$ & $-$ & $-$ & $183.4744810$ & $216.9142113$ & $225.3053712$ & $227.0563318$ \\
\hline
$\lambda_{13}$ & $-$ & $-$ & $198.6917619$ & $242.4962397$ & $253.7773704$ & $256.2351380$ \\
\hline
$\lambda_{14}$ & $-$ & $-$ & $212.4063351$ & $268.4292940$ & $283.2100111$ & $286.5368179$ \\
\hline
\end{tabular}
\caption{Approximation of the eigenvalues using method \ref{Discretization1:SL}.}\label{GL1}
\end{table}

\begin{figure}\label{ef_luisa}
\begin{tabular}{|c|c|}
\hline
\scalebox{0.6} {\includegraphics{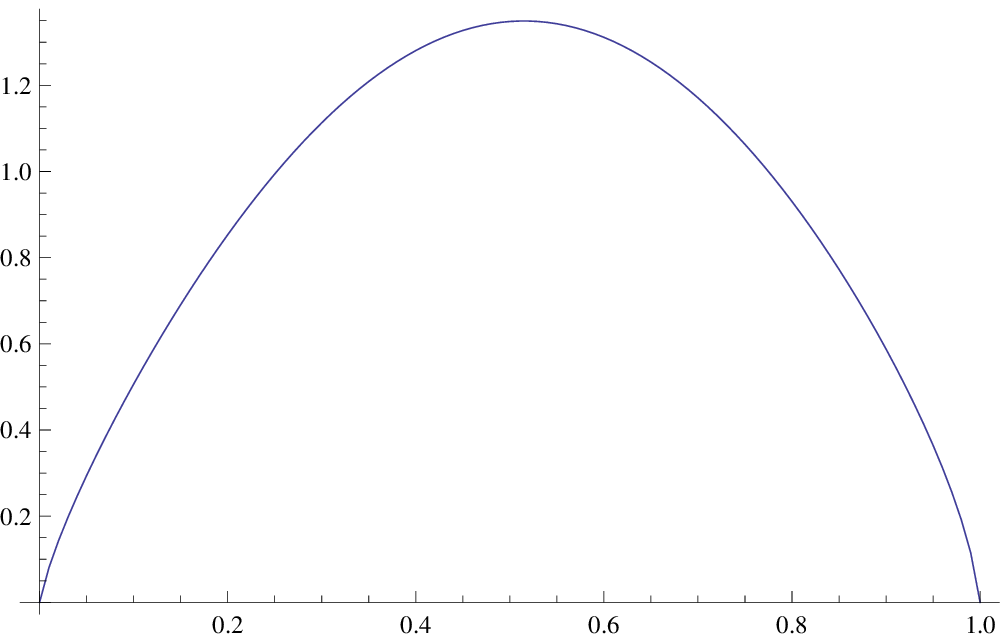}} & \scalebox{0.6} {\includegraphics{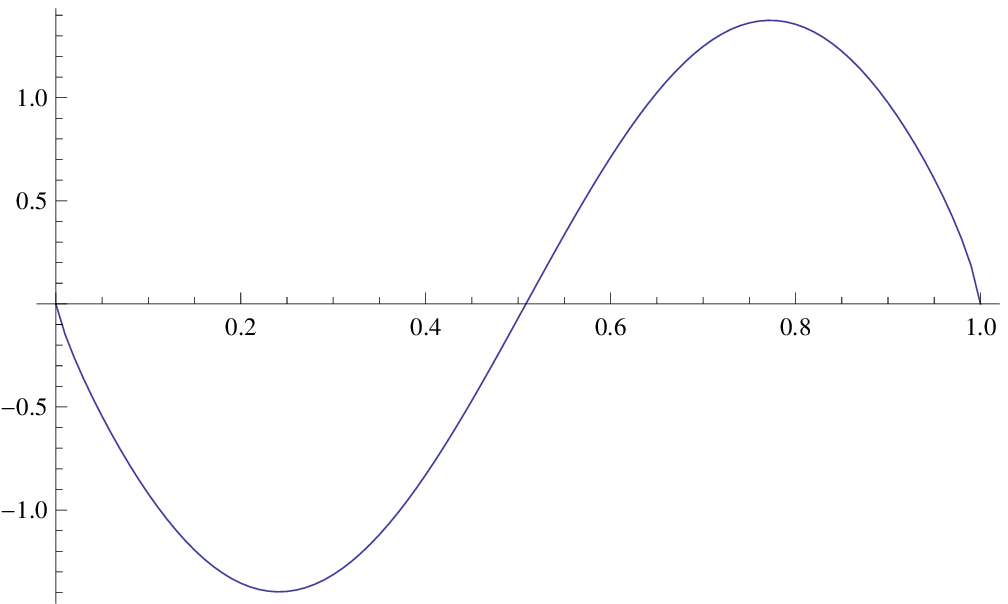}}\\
\hline
\scalebox{0.6} {\includegraphics{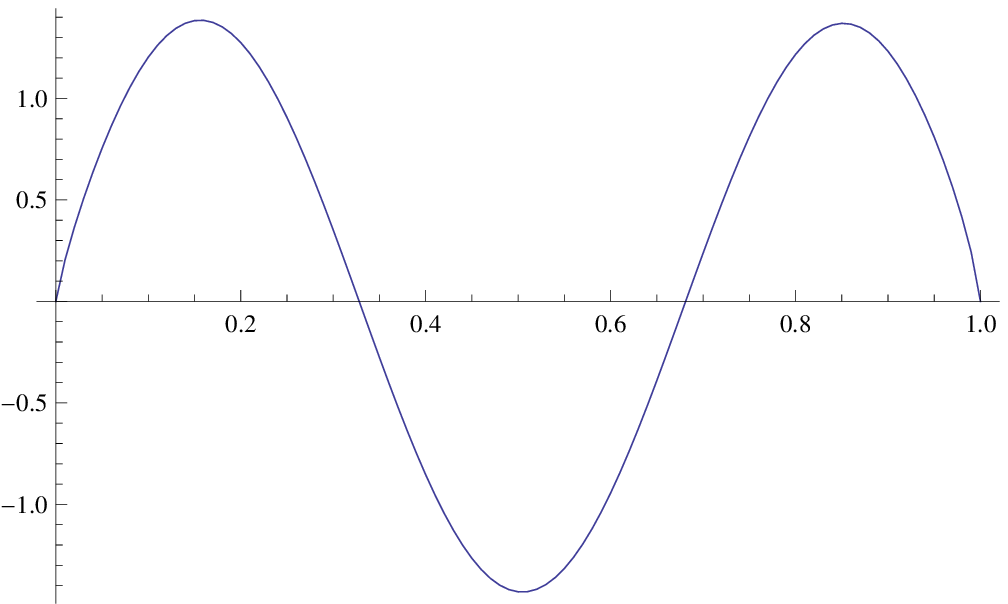}} & \scalebox{0.6} {\includegraphics{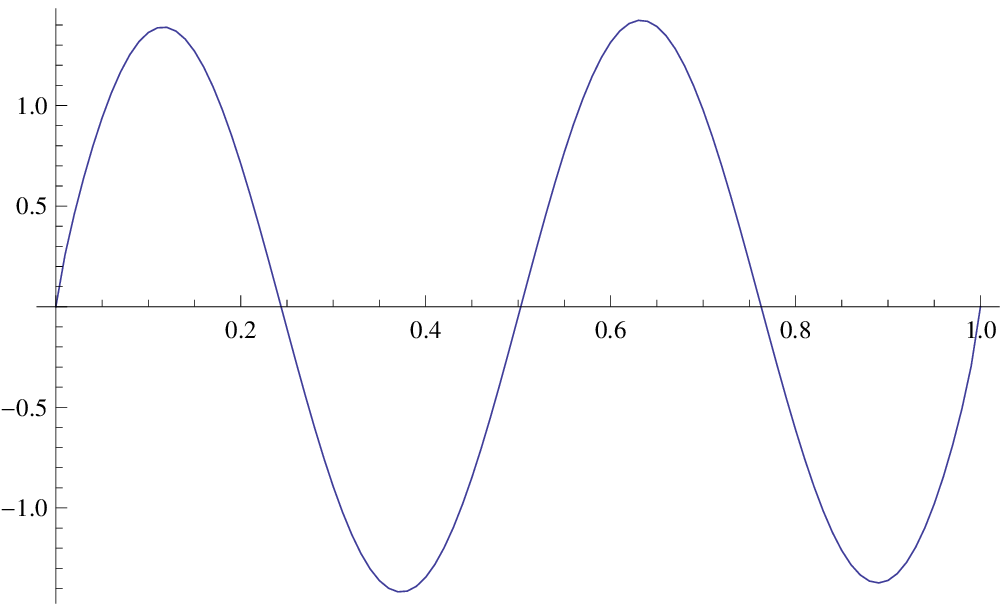}}\\
\hline
\end{tabular}
\caption{Normalized eigenfunctions obtained with $N=100$, corresponding to the  eigenvalues $\lambda_1$, $\lambda_2$, $\lambda_3$ and $\lambda_4$ (top left, top right, bottom left and bottom right, respectively).}
\end{figure}

\end{example}

\section{Conclusions}

Since 1986, when the seminal works were published \cite{Nigmat,Wyss},
fractional differential equations have become a popular way to model anomalous diffusion. As it is stated in \cite{MMM}
this type of approach is the most reasonable: the fractional derivative in space codes large particle jumps (that lead to anomalous super-diffusion) while the time-fractional derivative models time delays between particle motion. Fractional diffusion equations have been used, e.g., to model pollution in ground water \cite{ben} and flow in porous media \cite{He}. Many other examples can be found in \cite{MMM,MMM2}.

It was proved in \cite{Tatiana2} that, under appropriate assumptions,
the following space--time fractional diffusion equation
\begin{equation}\label{eq:Dif:1}
{}^{C}D^{\beta}_{0+,t} u(t,x)=-\frac{1}{r_{\alpha}(x)}\left[{}^{C}D^{\alpha}_{b-,x}p(x){}^{C}D^{\alpha}_{a+,x}+q(x)\right] u(t,x),\textnormal{ for all }(t,x) \in (0,\infty)\times [a,b],
\end{equation}
where, $0<\beta<1$, $\frac{1}{2}<\alpha<1$, and ${}^{C}D^{\beta}_{0+,t}$, ${}^{C}D^{\alpha}_{b-,x}$, ${}^{C}D^{\alpha}_{a+,x}$ are partial fractional derivatives, with the boundary and initial conditions:
\begin{equation}\label{eq:BC:Dif:1}
u(t,a)=u(t,b)=0,~~t\in (0,\infty),
\end{equation}
\begin{equation}\label{eq:IC:Dif:1}
u(0,x)=f(x),~~x\in [a,b],
\end{equation}
 has a continuous solution $u:[0,\infty)\times [a,b]\rightarrow\R$ given by the series
\begin{equation}\label{eq:sol:Diff:11}
u(t,x)=\sum\limits_{k=1}^{\infty} \langle y_k,f\rangle E_{\beta}(-\lambda_k t^{\beta})y_k(x).
\end{equation}
In \eqref{eq:sol:Diff:11}: $\langle f,g \rangle:= \int_{a}^{b}r_{\alpha}(x)f(x) g(x) \; dx,$ $E_{\beta}$ the one-parameter Mittag--Leffler function, $y_k$ and $\lambda_k$ ($k=1,2,\ldots$) are the eigenfunctions and the eigenvalues of the fractional Sturm--Liouville problem \eqref{eq:SLE}--\eqref{eq:BC}. Thus numerical methods, presented in this paper, for finding the eigenvalues and
eigenfunctions of fractional Sturm-Liouville problems can be also used
to approximate solution to fractional diffusion problems
of the form \eqref{eq:Dif:1}--\eqref{eq:IC:Dif:1}.
 We have presented a link between fractional Sturm--Liouville and fractional isoperimetric variational problems that provides a possible method for solution of those firstly mentioned problems. Discrete problems with the Gr\"{u}nwald--Letnikov difference were analyzed: we proved the existence of orthogonal solutions to the discrete fractional Sturm--Liouville eigenvalue problem and showed that its eigenvalues can be characterized as values of certain functionals. For continuous problems with the Caputo fractional derivatives, in order to examine the performance of the proposed method, the approximation based on the shifted Gr\"{u}nwald--Letnikov definition were used. This type of discretization is most popular in practical application, when solving numerically fractional diffusion equations, due to the fact that codes are mass-preserving \cite{Defterli}.

\section*{Acknowledgements}
Research supported by Portuguese funds through the CIDMA - Center for Research and Development in Mathematics and Applications, and the Portuguese Foundation for Science and Technology (FCT-Funda\c{c}\~ao para a Ci\^encia e a Tecnologia), within project UID/MAT/04106/2013 (R. Almeida), by
the Bialystok University of Technology grant S/WI/1/2016 (A. B. Malinowska), and
by the Warsaw School of Economics grant KAE/S15/35/15 (T. Odzijewicz).


\begin{thebibliography}{99}

\bibitem{abd}
T. Abdeljawad, On Riemann and Caputo fractional differences. Comp. and Math. with Appl. 62 Issue 3, 1602-1611 (2011).

\bibitem{Agrawal1} O. P. Agrawal,
Fractional variational calculus and the transversality conditions.
J. Phys. A: Math. Gen. 39, 10375--84 (2006).

\bibitem{AlM09}
Q. M. Al-Mdallal, An efficient method for solving fractional
Sturm--Liouville problems, Chaos Solitons Fractals 40, 183--189 (2009).

\bibitem{QMQ}
Q. M. Al-Madallal,
On the numerical solution of fractional Sturm--Liouville problem,
Int. J. of Comput. Math. 87, No.~12, 2837--2845 (2010).

\bibitem{Almeida1} R. Almeida, D. F. M. Torres,
Necessary and sufficient conditions for the fractional calculus of variations with Caputo derivatives.
Commun. Nonlinear Sci. Numer. Simul. 16, 1490--1500 (2011).

\bibitem{cap1:book:frac:ICP2}
R. Almeida, S. Pooseh, D. F. M. Torres,
 \emph{Computational methods in the fractional calculus of variations},
Imp Coll Press, London, 2015.

\bibitem{Atici1}
F. Atici,  P. W. Eloe, Initial value problems in discrete fractional calculus. Proc. Amer. Math.
Soc. 137,  981--989 (2009).


\bibitem{Atici2} F. Atici, P. Eloe, Discrete fractional calculus with the nabla operator. Elect. J. Qual. Theory
Differential Equations, Spec. Ed. I No. 3, 1--12 (2009).

\bibitem{ben}
D. A. Benson, R. Schumer, M. M. Meerschaert, S. W. Wheatcraft, Fractional dispersion, Levy motion, and the MADE tracer tests, Transp. Porous Media 42, 211-240 (2001).

\bibitem{Bl}
T. Blaszczyk, M. Ciesielski, Numerical solution of fractional Sturm--Liouville equation in integral form.
Fract. Calc. Appl. Anal. 17, No. 2, 307--320 (2014).


\bibitem{loic1}
L. Bourdin, J. Cresson, I. Greff, P. Inizan,
Variational integrator for fractional Euler--Lagrange equations. Appl. Numer. Math. 71, 14--23 (2013).

\bibitem{Carpinteri}
A. Carpinteri, F. Mainardi,
{\it Fractals and fractional calculus in continuum mechanics},
CISM Courses and Lectures, 378, Springer, Vienna, 1997.


\bibitem{Chen}
Z. Chen, M. M. Meerschaert, E. Nane,
Space--time fractional diffusion on bounded domains.
J. Math. Anal. Appl. 393, 479--488 (2012).

\bibitem{Defterli}
O. Defterli, M. D'Elia, Q. Du, M. Gunzburger, R. Lehouc, M. M. Meerschaert,
Fractional diffusion on bounded domains.
Fract. Calc. Appl. Anal. 18, No. 2, 342--360 (2015).

\bibitem{Domek}
S. Domek, P. Dworak (eds.), {\it Theoretical developments and applications of non--integer order systems}.
Lecture Notes in Electrical Engineering, Springer, vol. 357, 2015.

\bibitem{diaz}
 J. B. Diaz and T. J. Osler, Differences of fractional order. Math. Comp. 28, 185--201 (1974).

\bibitem{ovido}
M. D'Ovidio, From Sturm-Liouville problems to fractional and anomalous diffusions,
Stochastic Processes Appl. 122, No. 10,  3513--3544 (2012).

\bibitem{He}
J-H. He, Approximate analytical solution for seepage flow with fractional derivatives in porous media, Comput. Methods Appl. Mech. Engrg. 167, 57--68 (1998).

\bibitem{book:Hilfer}
R. Hilfer,
{\it Applications of fractional calculus in physics},
World Sci. Publishing, River Edge, NJ, 2000.

\bibitem{kaczorek}
T. Kaczorek, {\it Selected problems of fractional systems theory}, Lecture Notes in Control and
Information Sciences, vol. 411, Springer, Berlin, 2011.


\bibitem{Kilbas}
A. A. Kilbas, H. M. Srivastava, J. J. Trujillo,
{\it Theory and applications of fractional differential equations},
North-Holland Mathematics Studies, 204, Elsevier, Amsterdam, 2006.

\bibitem{cap1:book:Klimek}
M. Klimek,
  \emph{On solutions of linear fractional differential equations of a variational type},
The Publishing Office of Czestochowa University of Technology, Czestochowa, 2008.

\bibitem{Tatiana1}
M. Klimek, T. Odzijewicz, A. B. Malinowska,
Variational methods for the fractional Sturm-Liouville problem.
J. Math. Anal. Appl. 416, 402--426 (2014).

\bibitem{K15}
M. Klimek, Fractional Sturm-Liouville problem and 1D space-time fractional diffusion problem with mixed boundary conditions,
In: {\emph Proceedings of the ASME 2015 International Design Engineering Technical Conferences (IDETC) and Computers and Information in Engineering Conference (CIE)} 2015 Boston USA. Paper DETC2015-46808.

\bibitem{K16}
M. Klimek, Fractional Sturm-Liouville Problem in Terms of Riesz Derivatives. In:
\emph {Theoretical Developments and Applications of Non-Integer Order Systems} (Editors: Stefan Domek, Pawe\l~Dworak) 3--16, Series: Lecture Notes in Electrical Engineering 357, Springer International Publishing, Heidelberg 2016.

\bibitem{Tatiana2}
M. Klimek,  A. B. Malinowska, T. Odzijewicz, Applications of fractional Sturm-Liouville Problem to the space--time fractional diffusion in a finite domain, Fract. Calc. Appl. Anal., accepted (2016).

\bibitem{Leonenko}
N. N. Leonenko, M. M. Meerschaert, A. A. Sikorskii,
Fractional Pearson diffusion. J. Math. Anal. Appl. 403, 737--745 (2013).

\bibitem{Li}
J. Li, M. Ostoja-Starzewski,
Micropolar continuum mechanics of fractal media.
Internat. J. Engrg. Sci. 49, 1302--1310 (2011).

\bibitem{book:Mainardi}
F. Mainardi,
{\it Fractional calculus and waves in linear viscoelasticity},
Imp. Coll. Press, London, 2010.

\bibitem{cap1:book:AD}
A. B. Malinowska, D. F. M. Torres,
  \emph{Introduction to the fractional calculus of variations},
Imp Coll Press, London, 2012.

\bibitem{cap1:book:ADT}
A. B. Malinowska, T. Odzijewicz, D. F. M. Torres,
  \emph{Advanced methods in the fractional
calculus of variations}, SpringerBriefs in Applied Sciences and Technology, 2015.


\bibitem{malodz}
A. B. Malinowska, T. Odzijewicz, Multidimensional discrete-time fractional calculus of variations, Theoretical Developments and Applications of Non-Integer Order Systems [7th Conference on Non-integer Order Calculus and Its Applications], eds. Stefan Domek, Pawe³ Dworak, Lecture Notes in Electrical Engineering , vol. 357, Springer 2015, 17--28.

\bibitem{MMM}
M. M. Meerschaert, Fractional Calculus, Anomalous Diffusion, and Probability. {\it Fractional Dynamics}, World Scientific Publishing Co. Pte. Ltd., 265--284 (2011).

\bibitem{MMM2}
M. M. Meerschaert, A. Sikorskii,  \emph{Stochastic models for fractional calculus}, Walter de Gruyter, Berlin, 2012.

\bibitem{Metzler}
R. Metzler, J. Klafter, The random walk's guide to anomalous diffusion: a fractional dynamics approach,
Phys. Rep. 339, 1--77 (2000).

\bibitem{miller}
K. S. Miller and B. Ross, Fractional difference calculus, Proceedings of the International
Symposium on Univalent Functions, Fractional Calculus and Their Applications,
Nihon University, Koriyama, Japan, May (1988), 139-152; Ellis Horwood
Ser. Math. Appl., Horwood, Chichester, 1989.

\bibitem{Nigmat}
R. R. Nigmatullin,
The realization of the generalized transfer equation in a medium with fractal geometry,
Phys. Stat. Sol. B 133, No. 1, 425--430 (1986).

\bibitem{ostalczyk}
P. Ostalczyk, {\it Zarys rachunku r\'o\.zniczkowo ca\l kowego u\l amkowych rzed\'ow. Teoria i zastosowania w automatyce}
The Publishing office of \L \'od\'z University of Technology, \L\'od\'z (2008).

\bibitem{book:Podlubny}
I. Podlubny,
\textit{Fractional differential equations},
Mathematics in Science and Engineering, 198,
Academic Press, San Diego, CA, 1999.

\bibitem{Pooseh1}
S. Pooseh, R. Almeida and D. F. M. Torres,
Discrete direct methods in the fractional calculus of variations.
Comput. Math. Appl. 66, 668--676 (2013).

\bibitem{CD:Riewe:1996}
F. Riewe,
Non-conservative Lagrangian and Hamiltonian mechanics,
Phys. Rev. E (3) 53, 1890--1899 (1996).

\bibitem{CD:Riewe:1997}
F. Riewe,
Mechanics with fractional derivatives,
Phys. Rev. E (3)  55, 3581--3592 (1997).

\bibitem{Tarasov}
V. E. Tarasov, {\it Fractional Dynamics. Applications of fractional calculus to dynamics of particles, Fields and Media},
Higher Education Press, Beijing and Springer--Verlag, Berlin, Heidelberg, 2010.

\bibitem{Wyss}
W. Wyss, The fractional diffusion equation, J. Math. Phys.27, No. 11, 2782--2785 (1986).

\bibitem{zay}
M. Zayernouri, G. E. Karniadakis, Fractional Sturm-Liouville eigen-problems: theory
and numerical approximation, J. Comput. Phys. 252, 495--517 (2013).

\bibitem{Zaslavsky}
G. M. Zaslavsky,
{\it Hamiltonian chaos and fractional dynamics},
reprint of the 2005 original,
Oxford Univ. Press, Oxford, 2008.

\bibitem{Edelman}
G. M. Zaslavsky, M. A. Edelman,
Fractional kinetics: from pseudochaotic dynamics to Maxwell's demon,
Phys. D 193, 128--147 (2004).




\end{thebibliography}
\end{document}